\documentclass[12pt]{amsart}
\usepackage{amsmath, amssymb, amsthm, amsfonts, setspace, color,ulem,fullpage}
\usepackage{graphicx}
\advance \topmargin by -\headheight
\advance \topmargin by -\headsep

\evensidemargin \oddsidemargin
\newcommand{\lf}{\lfloor}
\newcommand{\llf}{\left\lfloor}
\newcommand{\rf}{\rfloor}
\newcommand{\rrf}{\right\rfloor}

\newcommand{\nf}{\llf\frac{n}{2}\rrf}

\newcommand{\mb}{\mathbb}

\newcommand{\Z}{\mb{Z}}

\newtheorem{theorem}{Theorem}
\newtheorem{lemma}[theorem]{Lemma}
\newtheorem{proposition}[theorem]{Proposition}
\newtheorem{conjecture}[theorem]{Conjecture}
\newtheorem{corollary}[theorem]{Corollary}
\newtheorem{remark}[theorem]{Remark}

\begin{document}
\title{Weighted Pebbling Numbers of Graphs}
\author{Stephanie Jones}
\address{Department of Mathematics\\
         Willamette University\\
        900 State St. \\
         Salem\\
         OR  97301
}
\email{stephanie3.1415@gmail.com}

\author{Joshua D. Laison}
\address{Department of Mathematics\\
         Willamette University\\
         900 State St. \\
         Salem\\
         OR  97301
}
\email{jlaison@willamette.edu}

\author{Cameron McLeman}
\address{Department of Mathematics\\
         The University of Michigan-Flint\\
        303 E. Kearsley Street\\
         Flint\\
         MI 48502
}
\email{mclemanc@umflint.edu}

\author{Kathryn Nyman}
\address{Department of Mathematics\\
         Willamette University\\
         900 State St. \\
         Salem\\
         OR  97301
}
\email{knyman@willamette.edu}

\thanks{Funded in part by a Willamette University College Colloquium Student Research Grant}

\keywords{pebbling number, weighted pebbling number}
\subjclass[2000]{Primary  05C57, 05C22; Secondary: 05C05, 05C38}
\date{\today}

\begin{abstract}
We expand the theory of pebbling to graphs with weighted edges.  In a weighted pebbling game, one player distributes a set amount of weight on the edges of a graph and his opponent chooses a target vertex and places a configuration of pebbles on the vertices.  Player one wins if, through a series of pebbling moves, he can move at least one pebble to the target. A pebbling move of $p$ pebbles across an edge with weight $w$ leaves $\lfloor pw \rfloor$ pebbles on the next vertex.  We find the weighted pebbling numbers of stars, graphs with at least $2|V|-1$ edges, and trees with given targets.
We give an explicit formula for the minimum total weight required on the edges of a length-2 path, solvable with $p$ pebbles and exhibit a graph which requires an edge with weight 1/3 in order to achieve its weighted pebbling number.   
\end{abstract}

\maketitle

\section{Introduction}
Graph pebbling originated two decades ago as a technique to prove a result on zero-sum sequences in finite groups \cite{chung89}.  Since then the subject has come into its own as an active area of research in graph theory, with over 50 papers by 80 authors \cite{hurlbert99, hurlbert05}.  
Problems in the field, guiding current research, include Graham's Conjecture on pebbling numbers of graph products \cite{feng01, herscovici03, herscovici98, wang09}, and pebbling numbers of graphs with small diameter \cite{bekmetjev09, bukh06, clarke97}.

One interpretation of the pebbling number arises from the \textit{pebbling game} we play against an evil opponent.  In this game, the evil opponent distributes some number of pebbles on the vertices of a graph $G$ and chooses a target vertex. Then we make pebbling moves on $G$. A pebbling move across an edge decreases the number of pebbles by a factor of 1/2, rounded down (see Figure \ref{pebbling moves}).  If we can reach the target with one or more pebbles, we win the game.  The \textit{pebbling number} of $G$ is the smallest number of pebbles that we can give the evil opponent and always win.

We can imagine a pebbling move across an edge $uv$ as multiplying the number of pebbles being moved from $u$ to $v$ by $1/2$.  In this way, $1/2$ can be thought of as the weight or cost on edge $uv$.  We generalize this idea to weighted graphs so that a pebbling move from $u$ to $v$ multiplies the number of pebbles being moved by the weight on edge $uv$, rounded down.   
As before, we can reinterpret these definitions game-theoretically.  Namely, we introduce the weighted pebbling game by adding a rule to the pebbling game: at the beginning of the game, we start with weights of $1/2$ on all the edges of $G$, and we are allowed to redistribute the weights on the edges of $G$ before handing the evil opponent the pebbles.  
The {\it weighted pebbling number} of $G$ is the minimum pebbling number over all possible distributions of this weight on $G$.

We begin Section 2 by defining ordinary pebbling numbers and stating a result by Chung. In Section 3, we define the weighted pebbling number and establish results for graphs with large numbers of edges.  We also give a formula for the pebbling number of weighted trees with determined targets and for the weighted pebbling number of stars.
Our main result appears in Section 4 in which we explore the minimum total weight required on a path to successfully pebble across it.
We prove the formula for this {\it pebbling weight function} for the path of length 2.
In Section 5, we find a graph that achieves its weighted pebbling number only when an edge with weight 1/3 is used.  We end with further questions and a lower bound on the pebbling weight function for a path of length $n$.

\section{Pebbling Numbers}

Following Hurlbert \cite{hurlbert05}, a \textit{\textbf{configuration}} $C$ of pebbles on a finite connected graph $G = (V,E)$ is a function $C : V \to \mathbb{N}$. The value $C(v)$ can be thought of as a number of pebbles placed at vertex $v$, and the \textbf{size} of the configuration is the sum $\sum_{v \in V} C(v)$ of pebbles on $G$. A \textit{\textbf{pebbling step}} along an edge from $u$ to $v$ takes $2$ pebbles from $u$ and adds $1$ pebble to $v$. We say that a vertex $t$ can be \textit{\textbf{reached}} by $C$ if one can repeatedly apply pebbling steps so that, in the resulting configuration $C'$, we have $C'(t)\geq 1$ (and $C'(v) \geq 0$ for all $v$). The \textit{\textbf{pebbling number}} $p(G)$ is the smallest integer $m$ so that any specified target vertex $t$ of $G$ can be reached by every configuration $C$ of size $m$. A configuration that can reach every vertex is called \textit{\textbf{solvable}}, and \textit{\textbf{unsolvable}} otherwise.  Figure \ref{pebbling moves} shows an example of a solvable pebbling configuration and its solution to an example (circled) target.

\begin{figure} [!htbp]
\begin{center}
  \includegraphics[width=7cm]{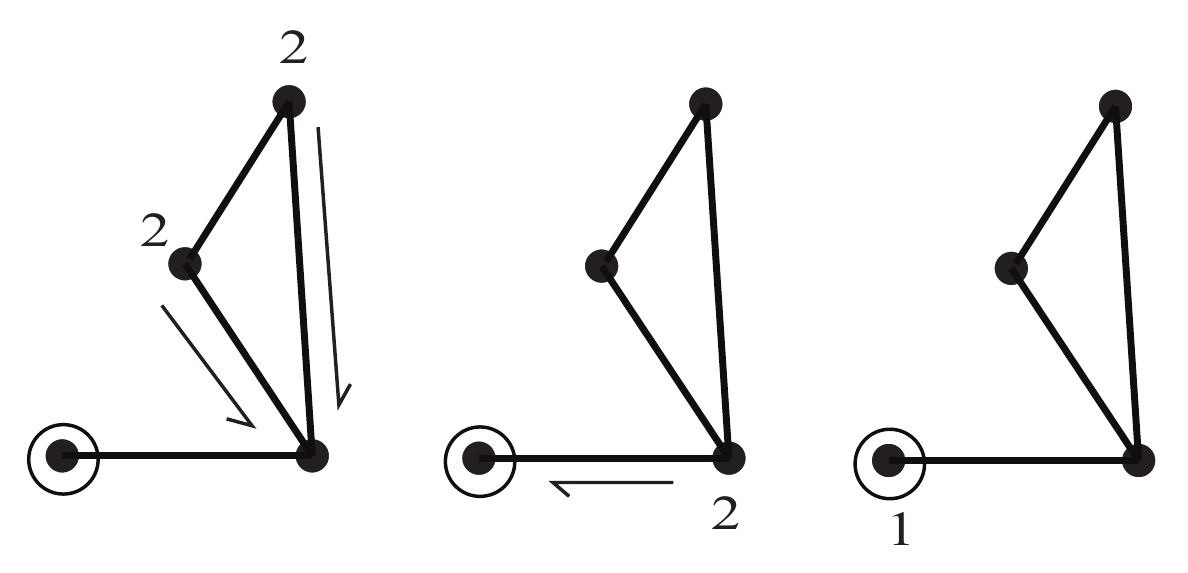}
  \end{center}
\caption{Pebbling to a target.}
 \label{pebbling moves}
\end{figure}

The pebbling number of a path with $n$ edges is $2^n$ \cite{hurlbert99}, and the pebbling number of a tree is given in the following theorem by Chung \cite{chung89}. A \textit{\textbf{path partition}} of a tree $T$ is a partition of the edges of $T$ into paths. A \textit{\textbf{maximum path partition}} is a path partition $P=\{p_1,p_2,...,p_m\}$ of $T$ such that $p_1$ is a maximal path in $T$, $p_2$ is a maximal path in $T-E(p_1)$, and so on. An example of a maximum path partition is given in Figure \ref{max partition}.

\begin{figure} [!htbp]
\begin{center}
  \includegraphics[width=9cm]{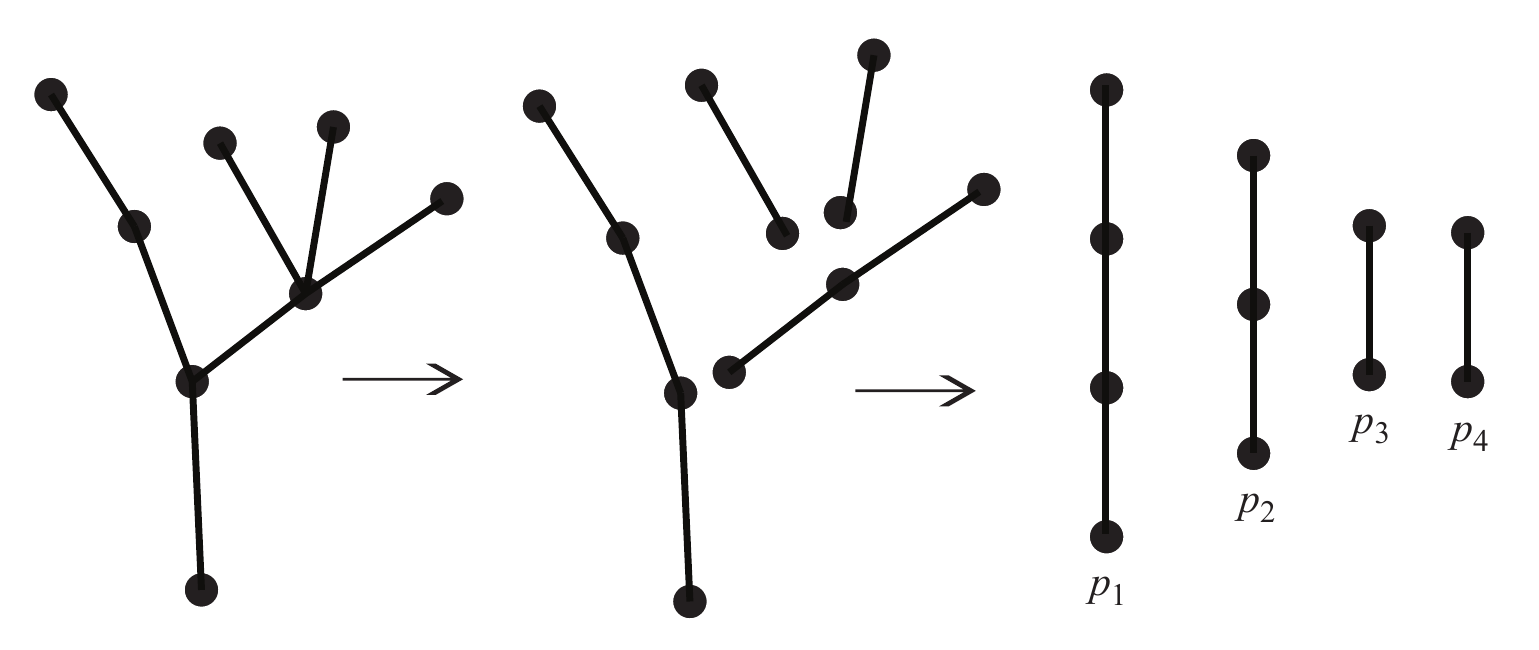}
  \end{center}
\caption{A maximum path partition of a tree.}
 \label{max partition}
\end{figure}

\begin{theorem}[Chung, 1989] \label{chung}
Let $T$ be a tree, and let $P=\{p_1,p_2,...,p_m\}$ be any maximum path partition of $T$, where each path $p_i$ has $\ell_i$ edges. Then $$ p(T)=\left( \sum_{i=1}^m p(p_i) \right) - m + 1= \left( \sum_{i=1}^m 2^{\ell_i} \right) - m + 1.$$ \end{theorem}  \noindent For example, the pebbling number of the graph in Figure \ref{max partition} is $(2^3+2^2+2^1+2^1)-4+1=13$.

\section{Weighted Pebbling Numbers}

If $G$ is a weighted graph with a weight of $1/2$ on each edge, then we can think of the pebbling step from $u$ to $v$ as multiplying the number of pebbles being moved by the weight on the edge $uv$, rounding down.  We generalize this idea to say that if $uv$ has some other edge weight $w$ between 0 and 1, the \textit{\textbf{pebbling step}} along $uv$ removes $k$ pebbles from $u$ and places $\lfloor w k \rfloor$ pebbles on $v$, for some positive integer $k$.

Formally, given a graph $G$, a \textit{\textbf{weight distribution}} $W$ on $G$ is an assignment of weights to the edges of $G$.  
For each edge $e$, we denote its weight by $w_e$, and we denote the corresponding weighted graph by $G_W$.  We require $0 \leq w_e \leq 1$ for all edge weights $w_e$.   The \textit{\textbf{total weight}} $|W|$ of $W$ is $\sum_{e \in E(G)} w_e$.

A weighted graph $G_W$ is \textit{\textbf{$p$-solvable}} if every configuration of $p$ pebbles on $G$ is solvable to every target vertex.  
The \textit{\textbf{weighted pebbling number}} of a weighted graph $wp(G_W)$ is the smallest positive integer $p$ for which $G_W$ is $p$-solvable.  Figure \ref{distribute} shows that different choices of weight distributions on the same graph yield different pebbling numbers.  An unweighted graph $G$ is \textbf{\textit{$(W,p)$-solvable}} if $G_W$ is $p$-solvable, and $G$ is \textbf{\textit{$(w,p)$-solvable}} for some positive real number $w$ if there exists a weight distribution $W$ on $G$ with $|W|=w$ such that $G$ is $(W,p)$-solvable.  The \textit{\textbf{weighted pebbling number}}, $wp(G)$, of an unweighted graph $G$ is the smallest positive integer $p$ for which $G$ is $(|E(G)|/2, p)$-solvable.  Equivalently, $wp(G)= \min_{|W|=E(G)/2} wp(G_W)$.  In other words, the weighted pebbling number of $G$ is the smallest number $p$ of pebbles such that there exists a weight distribution on $G$ of weight $|E(G)|/2$ for which all configurations of $p$ pebbles on $G$ are solvable.  

\begin{figure} [!htbp]
\begin{center}
  \includegraphics[width=10cm]{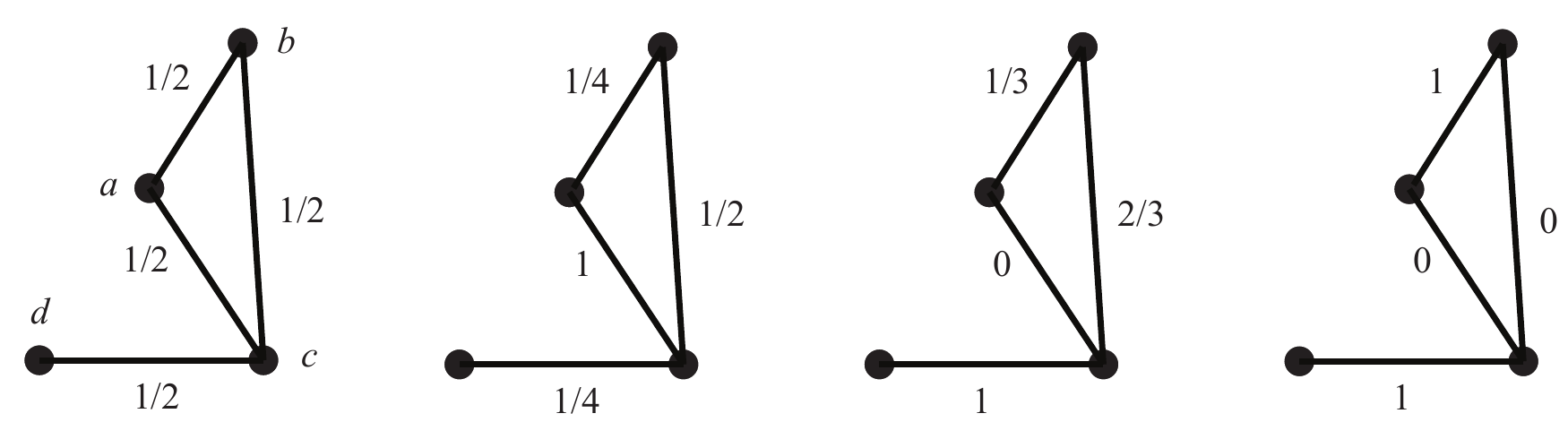}
  \end{center}
\caption{A graph with four different weight distributions.  The first three weighted graphs have weighted pebbling numbers 5, 8, and 6, respectively (Note for example in the first graph that 3 pebbles on vertex $a$, and 1 pebble on $b$ is not solvable to $d$).  The last weight distribution is not $p$-solvable for any $p$.}
 \label{distribute}
\end{figure}

Since placing weights of $1/2$ on all the edges yields the standard pebbling number, $wp(G) \leq p(G)$ for all graphs $G$.  On the other hand, the pebbling number and weighted pebbling number of $G$ can be arbitrarily far apart.  For example, $p(K_n)=n$ \cite{hurlbert99}, but $wp(K_n)=1$ for all $n \geq 4$.  Figure \ref{k4} shows a weight distribution on $K_4$ which is $1$-solvable.  More generally, the following proposition shows that $wp(G)=1$ for all graphs with $n$ vertices and at least $2n-2$ edges.

\begin{figure} [!htbp]
\begin{center}
  \includegraphics[width=2.5cm]{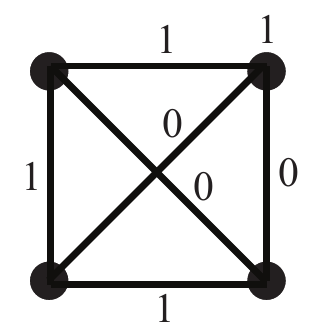}
  \end{center}
\caption{A weight distribution and pebbling configuration on $K_4$ that shows $wp(K_4)=1$.}
 \label{k4}
\end{figure}

\begin{proposition} \label{2n-2}
Let $G$ be a connected graph with $n$ vertices.  Then $wp(G)=1$ if and only if $G$ has at least $2n-2$ edges.
\end{proposition}

\begin{proof}
Suppose $G$ has at least $2n-2$ edges.  Let $T$ be a spanning tree of $G$, and assign each edge of $T$ a weight of 1.  Then assign any remaining weight arbitrarily to the other edges of $G$, keeping the constraint $0 \leq w_e \leq 1$.  
This weight distribution has pebbling number 1, since any single pebble can move from any vertex of $G$ to any other vertex along a path through $T$ with no loss of pebbles.  Finally, no configuration of 0 pebbles is solvable, so $wp(G)=1$.

Conversely, suppose $wp(G)=1$.  Then there exists a weight distribution on $G$ for which every configuration of 1 pebble is solvable to every target.  Then for this weight distribution, there exists a path between every pair of vertices of $G$ consisting of edges with weight 1.  This requires at least $n-1$ edges of weight 1, so this weight distribution has total weight at least $n-1$.  Thus, $G$ has at least $2n-2$ edges.
\end{proof}

Note that if $G_W$ is a weighted graph, then $wp(G_W)=wp(G'_W)$, where $G'_W$ is the graph obtained from $G_W$ by deleting every edge of $G_W$ with weight 0 and contracting every edge of $G_W$ with weight 1.

\begin{proposition} \label{2n-k}
For any connected graph $G$ with $n$ vertices and $2n-2-k$ edges, where $k$ is a positive integer, $wp(G)\leq 2^k$.
\end{proposition}

\begin{proof}
Let $T$ be a spanning tree of $G$, and let $W$ be a weight distribution on $G$ with weight 1 on $n-1-k$ arbitrarily chosen edges of $T$, weight $1/2$ on the remaining $k$ edges of $T$, and weight $0$ on the remaining edges of $G$.  Then $|W|=n-1-k+k/2= (2n-2-k)/2 = |E(G)|/2$.  Let $G'$ be the graph obtained from $G_W$ by deleting the edges with weight 0 and contracting the edges with weight 1.  Since the remaining edges of $G_W$ have weight $1/2$, $wp(G_W)=p(G')$.  Since $G'$ is a tree, its pebbling number is given by Theorem~\ref{chung} above.  This formula is bounded above by the pebbling number of a path with the same number of edges, so $wp(G_W) \leq 2^k$. Finally, $wp(G)\leq wp(G_W)$ for any weight distribution $W$ of weight $|E(G)|/2$ on $G$.
\end{proof}

\noindent Proposition~\ref{starpf} will demonstrate that the bound in Proposition~\ref{2n-k} is not sharp.

\begin{corollary} \label{2n-3}
For any connected graph $G$ with $n$ vertices and $2n-3$ edges, $wp(G)=2$.
\end{corollary}

\begin{proof}
By Proposition~\ref{2n-2}, $wp(G)>1$.  By Proposition~\ref{2n-k}, $wp(G) \leq 2$.
\end{proof}

Proposition~\ref{2n-2} tells the complete story for weighted pebbling numbers of connected graphs with a large number of edges, so we turn our attention for the remainder of the paper to weighted pebbling numbers of connected graphs with a small number of edges, in particular paths and trees.  

The \textit{\textbf{requirement}} of an edge $e=uv$ is the smallest number of pebbles needed on $u$ to move at least one pebble across $uv$ to $v$.  If $e$ has edge weight $w_e$, then $e$ has requirement $r_e= \lceil 1/w_e \rceil$.  The \textit{\textbf{requirement}} $r(P)$ of a directed path $P$ with edges $e_1=v_1v_2$ through $e_k=v_kv_{k+1}$ is the smallest number of pebbles needed on $v_1$ to move at least one pebble along $P$ to $v_{k+1}$.  If edge $e_i$ has weight $w_{e_i}$, then the requirement of $P$ is  $\lceil 1/w_{e_1} \lceil 1/w_{e_2} \cdots \lceil 1/w_{e_k} \rceil \rceil \cdots \rceil$.  Note that this number depends on the direction that $P$ is traversed.  For example, if $e_1=v_1 v_2$ and $e_2=v_2 v_3$ have weights $w_{e_1}=2/5$ and $w_{e_2}=1/2$, then $r(v_1 v_2 v_3)=5$ but $r(v_3 v_2 v_1)=6$.

Let $T$ be a weighted tree with designated target $t$.  A \textit{\textbf{maximum path partition}} of $T$ is a path partition $P=\{p_1,p_2,...,p_m\}$ of $T$ such that $p_1$ is a directed path in $T$ ending at $t$ with maximal requirement, $p_2$ is a directed path in $T-E(p_1)$ ending at a vertex of $p_1$ with maximal requirement, $p_3$ is a directed path in $T-E(p_1 \cup p_2)$ ending in $p_1 \cup p_2$ with maximal requirement, and so on.

Curtis, et.~al.~generalized Theorem~\ref{chung} for weighted graphs with integer edge weights corresponding to our edge requirements. This is equivalent to restricting our edge weights to be of the form $1/n$, for $n$ an integer.  We now extend their result to all rational edge weights.  
Let $wp(G, t)$ be the weighted pebbling number of a graph $G$ with a specified target $t$, and note that $wp(G)= \max_{t}[wp(G,t)]$.
The following is a modification of Theorem 6 \cite{curtis09}.

\begin{theorem}
\label{BigTheorem}
Suppose $T_W$ is a weighted tree with designated target $t$.  Let $P=\{p_1, \ldots, p_m\}$ be a maximum path partition of $T$ with respect to $W$ and $t$.  Then the weighted pebbling number given target $t$ is
$$wp(T_W, t)=\left( \sum_{i=1}^m r(p_i) \right) - m + 1.$$

\end{theorem}

\begin{proof}  Begin by replacing the edge weights on $T_W$ with the edge requirements.  In this context, Curtis, et.~al.~proved 
$wp(T_W, t)\leq \left( \sum_{i=1}^m r(p_i) \right) - m + 1$, (with $k=1$) \cite{curtis09}. 
We obtain the reverse inequality by noting that the configuration which places $r(p_i)-1$ pebbles on the initial vertex of directed path $p_i$, for all $1\leq i \leq m$, is not solvable to the target $t$. 
\end{proof}

\begin{corollary}
\label{PathPebblingNumber}
If $G$ is a path, $wp(G)=p(G)$.
\end{corollary}

\begin{proof}
Since a maximum path partition of $G$ is just $p_1=G$, by Theorem~\ref{BigTheorem}, $wp(G)=r(G) \geq \prod_{e \in E(G)} 1/w_e$.  This product is minimized if the weights $w_e$ are all equal, in which case we have $wp(G) \geq 2^{n-1}$.  Conversely, $wp(G) \leq p(G)=2^{n-1}$.
\end{proof}

As with paths, the weighted pebbling number of stars coincides with the pebbling number.  

\begin{remark}
In the upcoming Propositions \ref{starpf} and \ref{thm:1/3-tree} and in Theorem \ref{thm:1/3-required}, we use the results of a computer search.  
Consider a path $P=x_1 x_2 \ldots x_k$ with $p$ pebbles on $x_1$.  To find the minimum weight required to move at least one pebble from $x_1$ to $x_k$, the program searches through all $k$-tuples $(p=p_1, p_2, \ldots, p_{k-1},p_k=1) \in \Z^k$ with $p_i \geq p_{i+1}$ for all $i$, and returns the minimum sum, over all such $k$-tuples, of the minimum weight $p_{i+1}/p_i$ required on each edge.  

To find the minimum weight required to move $p$ pebbles from either $x_1$ to $x_k$ or from $x_k$ to $x_1$, 
the program looks at pairs of tuples, $(p=p_1, \ldots, p_k=1)$ and $(1=q_1, \ldots, q_k=p)$ which denote the intermediate pebbles as $p$ pebbles are moved from $x_1$ to $x_k$ and from $x_k$ to $x_1$ respectively.  The program chooses the maximum weight of the two possible required edge weights (one for each direction) and returns the minimal sum of these edge weights, over all pairs of such $k$-tuples.  This strategy is also used in the proof of Theorem \ref{p2-theorem}.
\label{rem:CompSearch}
\end{remark}

Let $S_k$ be the star with $k$ leaves.

\begin{proposition}\label{starpf}
$wp(S_k)=p(S_k)=k+2$ for all $k \geq 2$.
\end{proposition}

\begin{proof}
First, by Theorem~\ref{chung}, $p(S_k)=k+2$.  So $wp(S_k) \leq k+2$.

Suppose by way of contradiction that $wp(S_k)<k+2$.  So there exists a weight distribution $W$ with total weight $|E(G)|/2$ on $S_k$ with weighted pebbling number less than $k+2$.  We construct a maximal path partition $P=\{p_1, \ldots, p_j\}$ from $W$.  Note that every maximal path partition of $S_k$ has $k-1$ paths, one with two edges and the rest with one edge.  Since $p_1$ is the path with largest requirement in $S_k$ by definition, $p_1$ is the path consisting of the two edges in $S_k$ with smallest weight.  By Theorem~\ref{BigTheorem},
if all the edge weights in $W$ are $1/2$, then $wp(W)=k+2$, with $r(p_1)=4$ and $r(p_i)=2$ for all $1<i \leq k-1$.  Since $wp(W)<k+2$ by assumption, either the requirement of $p_1$ is less than $4$, or the requirement of one of the other paths is less than $2$.

By a computer search (see Remark~\ref{rem:CompSearch}), if $r(p_1)<4$, then the combined weight on the edges of $p_1$ is greater than 1.  Thus, the average weight of the edges of $p_1$ is greater than $1/2$, contradicting our choice of $p_1$.  
Suppose $r(p_i)<2$ for some $1< i \leq k-1$, and suppose the path $p_i$ is the edge $v_0 v_i$, where $v_0$ is the center vertex.  Since the requirement of any path must be at least $1$, $r(p_i)=1$.  So $v_0 v_i$ has weight $1$.  Suppose $m$ of the edges in $W$ have weight $1$, $m \geq 1$.  Then since no edge can have weight $0$, there must be at least $m+1$ edges in $S_k$ with weight less than $1/2$.  Call these edges $v_0 v_1$ through $v_0 v_{m+1}$.  Then the pebbling configuration with target $v_1$, 0 pebbles on $v_0$, $4$ pebbles on $v_2$, $2$ pebbles on each of $v_3$ through $v_{m+1}$, and $1$ pebble on each of the remaining vertices $v_j$ with $w(v_0 v_j)<1$ is not solvable, and has $4+2(m-1)+(k-m)=k+2+m>k+2$ pebbles since $m \geq 1$.
\end{proof}

\section{Minimum pebbling weights of paths of length two}

In this section we explore the minimum total weight required on the edges of a path of length two in order to successfully pebble $p$ pebbles across it.
Given a graph $G$, we define the \emph{\textbf{pebbling weight function}} $w_G$ on positive integers $p$ by
$$w_G(p)=\inf\{\,w\,|\, \text{$G$ is $(w,p)$-solvable}\}.$$
The following lemma shows this infimum is always obtained, so we can replace the infimum in the definition with
the minimum.  In other words, given a graph $G$ and a number of pebbles $p$, $w_G(p)$ is the minimum
weight $w$ such that there exists a weight distribution on $G$ of total weight $w$ for which $G$ is $p$-solvable.

\begin{lemma} \label{rational}
For any $G$ and $p$, there exists a weight distribution $W$ on $G$ with total weight
$w_G(p)$.  Moreover, such a $W$ has each edge weight $w_e$ a rational number, and hence $w_G(p)$ is rational.
\end{lemma}
\begin{proof}
We introduce a partial order on weight distributions $W$ such that $G$ is $(W,p)$-solvable by setting $W'\leq W$ if $w'_e\leq w_e$ for all $e\in E(G)$.  Suppose $G$ is $(W,p)$-solvable.  We construct a weight distribution $W'$ on $G$ by setting
$$w'_e=\max\left\{\frac{\lf aw_e\rf}{a},1\leq a\leq p\right\}$$
for each $e\in E(G)$.  We claim that $G$ is $(W',p)$-solvable.  To see this, suppose a pebbling step across $e$ starts with $c$ pebbles and ends with $d$ pebbles in $G_W$.  Then by definition $\lfloor c w_e \rfloor =d$, so $d=\lfloor c (\lf cw_e \rf / c) \rfloor \leq cw'_e$, so the same pebbling step ends with at least $d$ pebbles in $G_{W'}$.
So any solution of $G_W$ is also a solution of $G_{W'}$.

Finally, note that $W'$ has rational edge weights, each of which has a numerator and denominator bounded above by $p$.  Let us temporarily call such a distribution \emph{minimal}.  Since an arbitrary weight distribution is bounded below (with respect to our partial order above) by a minimal distribution, we conclude
$$
\inf\{\,w\,|\, \text{$G$ is $(w,p)$-solvable}\}=\inf\{\,|W|\,|\, \text{$G$ is $(W,p)$-solvable}\}= 
$$
$$
\inf\{\,|W|\,|\, \text{$W$ is minimal and $G$ is $(W,p)$-solvable}.\}
$$
But the last set is finite (in fact, of order at most $|E(G)|^{p^2}$) since there are only finitely many ways to place rational numbers with integral numerators and denominators bounded by $p$ on the edges of $G$.  Hence $wp(G)$ is the minimum weight of these minimal distributions.  
\end{proof}

We note that the weight distribution for the minimum pebbling weight of a graph displays some surprising behavior.  For example, the path $P_2$ with 2 edges is solvable in either direction with $p=6$, and weights 1/3 and 1/2 on the two edges. However, $P_2$ is not solvable with 6 pebbles if this same total weight is distributed evenly, with a weight of 5/12 on each edge.  This demonstrates the fact that the weighted pebbling number is not minimized, as one might expect, by distributing the weight equally between the two edges.
In Theorem~\ref{p2-theorem} we give an explicit formula for $w_{P_2}(p)$.  An explicit formula for $w_G(p)$ for any $G$ with a larger number of edges is still unknown.  

\begin{theorem} \label{p2-theorem}
Given a positive integer $p$, let $n=\lfloor\sqrt{p}\rfloor$.  Then
the pebbling weight function for the path graph $P_2$ with two edges
is given by
$$
w_{P_2}(p)=\begin{cases}
\frac{2}{n}&\text{ if }p\in[n^2,n^2+\nf]\\[.5em]
\frac{2n+1}{p}&\text{ if }p\in(n^2+\llf\frac{n}{2}\rrf,n^2+2\llf\frac{n}{2}\rrf]\\[.5em]
\frac{2n+1}{n(n+1)}&\text{ if }p\in(n^2+2\llf\frac{n}{2}\rrf,n^2+n+\llf\frac{n}{2}\rrf\big]\\[.5em]
\frac{2n+2}{p}&\text{ if }p\in (n^2+n+\llf\frac{n}{2}\rrf,n^2+2n].
\end{cases}
$$
\end{theorem}

\begin{proof}
Say the vertices of $P_2$ are $x$, $y$, and $z$, with leaves $x$ and $z$.  It suffices to provide a weight distribution in the two worst-case pebbling configurations: $C_1$ with $p$ pebbles on $x$ and target $z$, and $C_2$ with $p$ pebbles on $z$ and target $x$.  Now consider any weight distribution for which $P_2$ is solvable.  Such a weight distribution determines two integers $a$ and $b$, obtained from the solutions of $C_1$ and $C_2$ as follows.  From $C_1$, move all $p$ pebbles across $xy$ to obtain $a$ pebbles on $y$, and then move the $a$ pebbles across $yz$ to obtain at least $1$ pebble on $z$; from $C_2$, move all $p$ pebbles across $yz$ to obtain $b$ pebbles on $y$, and then move $b$ pebbles across $xy$ to obtain at least $1$ pebble on $x$ (see Figure \ref{fig:abpath}).

\begin{figure} [!htbp]
\begin{center}
 \includegraphics[width=6cm]{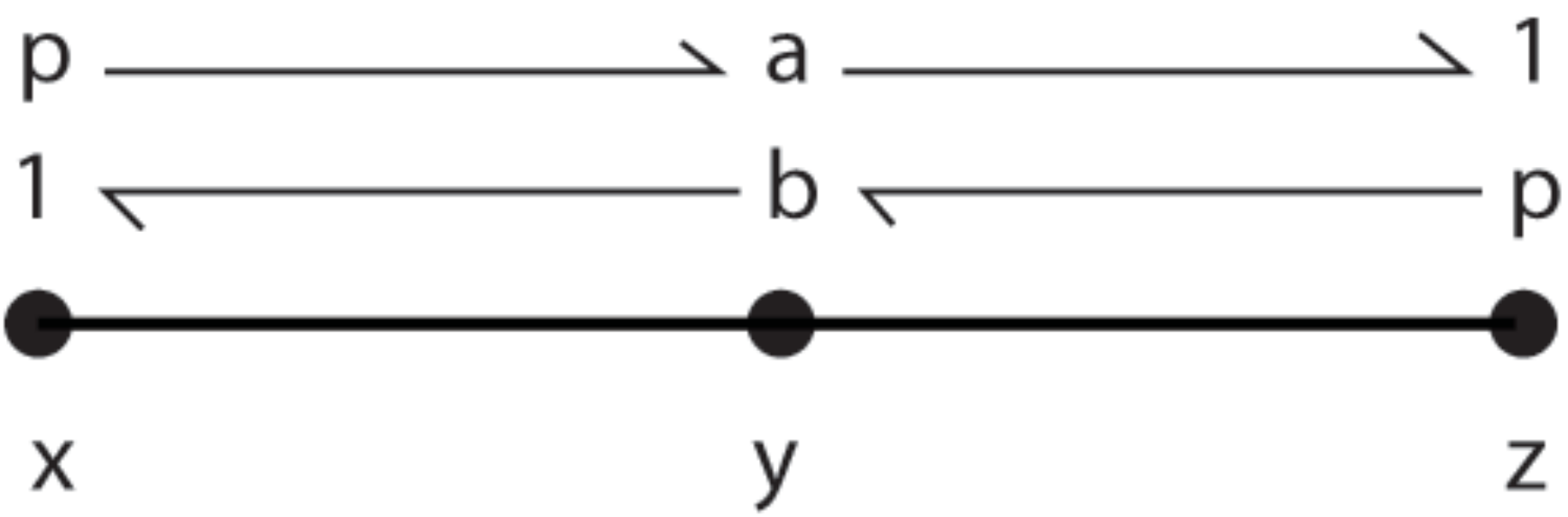}
  \end{center}
\caption{Two pebbling configuration solutions with intermediate steps.}
 \label{fig:abpath}
\end{figure}

\noindent The strategy for the proof is to find, for a given $a$ and $b$ between $1$ and $p$, the minimum total weight $w_{a,b}$ of such a weight distribution and then minimize $w_{a,b}$ over all such choices of $a$ and $b$.
\bigskip

We begin by computing $w_{a,b}$.  For a fixed choice of $a$ and $b$,
to ensure that
the rightward path in Figure~\ref{fig:abpath} is achieved as cheaply as possible, we weight the
first edge with $a/p$ and the second with $1/a$.  For
the leftward path, the minimizing weights are similarly $b/p$
and $1/b$.  Combined, we find that the smallest total weight for a given $a$ and $b$ is
$$
w_{a,b}=\max\left(\frac{a}{p},\frac{1}{b}\right)+\max\left(\frac{b}{p},\frac{1}{a}\right),
$$ and hence $w_{P_2}(p)$ is the minimum value this weight can take over
all possible choices $0\leq a,b\leq p$.  Note that
$$
\frac{a}{p}\geq \frac{1}{b}\text{ if and only if } \frac{b}{p}\geq\frac{1}{a},
$$ giving
$$w_{a,b}=\max\left(\frac{a+b}{p},\frac{1}{a}+\frac{1}{b}\right)=\max\left(\frac{a+b}{p},\frac{a+b}{ab}\right)=\frac{a+b}{\min(p,ab)}.$$
Finally, we conclude
\begin{align}
w_{P_2}(p)=\min_{0\leq a,b\leq p}\left(\frac{a+b}{\min(p,ab)}\right) = \min \left( \min_{0 \leq a,b \leq p, \, ab \leq p} \frac{a+b}{p} , \, \min_{0 \leq a,b \leq p, \, ab \geq p} \frac{a+b}{ab} \right).
\end{align}
It remains to compute this minimum explicitly for a given value of $p$.  We do
this by computing the minimums of the two functions on the right, subject
to their given constraints, and then taking the smaller of the two.  Lemmas~\ref{function1} and \ref{function2} below are precisely these calculations, and the result follows.
\end{proof}

\begin{lemma} \label{function1}
Let $f(a,b)=a+b$, and for a positive integer $p$, let $f_p$ denote the
minimum value of $f(a,b)/p$ with integer inputs $1\leq a\leq b\leq p$
satisfying $ab\geq p$.  Let $n=\lf\sqrt{p}\rf$.
Then
$$
pf_p=\begin{cases}
2n&\text{ if }p=n^2\\
2n+1&\text{ if }n^2<p\leq n^2+n\\
2n+2&\text{if }n^2+n<p\leq n^2+2n.
\end{cases}
$$
\end{lemma}
\begin{proof}
We verify the formula for each of the three cases.  Note that since $n=\lf\sqrt{p}\rf$ and $p$ is an integer, $n^2 \leq p \leq n^2+2n$, so these three cases include all possible values of $p$.
\bigskip

\noindent \textbf{Case 1. $p=n^2$.}  In this case $f$ achieves its minimum at $a=b=n=\lf\sqrt{p}\rf$, which agrees with the absolute minimum of the continuous function $f(a,b)=a+b$.
\bigskip

\noindent \textbf{Case 2. $n^2<p\leq n^2+n$.}  Note that if $a=n$ and $b=n+1$ then $f(a,b)=2n+1$.  These values of $a$ and $b$ satisfy $ab \geq p$ since $ab=n(n+1)=n^2+n\geq p$ by assumption.  Now assume by way of contradiction that there exist integers $c$ and $d$ with $1 \leq c \leq d \leq p$, $cd \geq p$, and $f(c,d)< 2n+1$.  Since $f(c,d)=c+d$ is an integer, we have $f(c,d)\leq 2n$.

\textbf{Subcase 2a. $c,d\leq n$.}  Then $cd\leq
\llf\sqrt{p}\rrf^2<p$, contradicting $cd\geq p$.

\textbf{Subcase 2b. $c,d\geq n+1$.} Then $c+d>2n+1$, contradicting $f(c,d)<2n+1$.

\textbf{Subcase 2c. $c\leq n$ and $d \geq n+1$.}  Write $c=n-j$ and
$d=n+k$ with $j \geq 0$ and $k\geq 1$.  Then since $c+d\leq 2n$ we have $k-j\leq 0$.  Therefore
$$cd=n^2+\underbrace{(k-j)}_{\leq 0}n+\underbrace{(-jk)}_{\leq 0}\leq n^2<p,$$
which contradicts $cd>p$.
\bigskip

\noindent \textbf{Case 3. $n^2+n<p\leq n^2+2n$.} Note that if $a=b=n+1$ then $f(a,b)=2n+2$.  These values of $a$ and $b$ satisfy $ab \geq p$ since  $ab=n^2+2n+1>n^2+2n\geq p$ by assumption.  Now assume by way of contradiction that there exist integers $c$ and $d$ with $1 \leq c \leq d \leq p$, $cd \geq p$, and $f(c,d)< 2n+2$.

\textbf{Subcase 3a. $c\leq n$ and $d\leq n+1$.}  Then $cd\leq n(n+1)=n^2+n<p$, contradicting $cd \geq p$.

\textbf{Subcase 3b. $c,d\geq n+1$.} Then $f(c,d)\geq 2n+2$, contradicting $f(c,d)<2n+2$.

\textbf{Subcase 3c. $c\leq n$ and $d>n+1$.} Write
$c=n+1-j$ and $d=n+1+k$ with $j,k\geq 1$.  Since $c+d< 2n+2$
we have $k-j<0$.  Therefore $$cd=n^2+\underbrace{(2+k-j)}_{\leq 1}n+\underbrace{(k-j)}_{<0}+\underbrace{(-jk)}_{<0}\leq n^2+n<p,$$
which contradicts $cd\geq p$.
\end{proof}

\begin{lemma} \label{function2}
Let $g(a,b)=1/a+1/b$, and for a positive integer $p$, let $g_p$ denote the
minimum value of $g(a,b)$ with integer inputs $1\leq a\leq b\leq p$
satisfying $ab\leq p$.  Let $n=\lf\sqrt{p}\rf$.
Then
$$
g_p
\begin{cases}
=\frac{2}{n}&\text{ if }p\in[n^2,n^2+\nf]\\[.5em]
\geq\frac{2n+1}{p}&\text{ if }p\in(n^2+\llf\frac{n}{2}\rrf,n^2+2\llf\frac{n}{2}\rrf]\\[.5em]
=\frac{2n+1}{n(n+1)}&\text{ if }p\in(n^2+2\llf\frac{n}{2}\rrf,n^2+n+\llf\frac{n}{2}\rrf\big]\\[.5em]
\geq\frac{2n+2}{p}&\text{ if }p\in (n^2+n+\llf\frac{n}{2}\rrf,n^2+2n].
\end{cases}
$$
\end{lemma}

\begin{proof}
We verify the formula for each of the four cases.  Note again that since $n=\lf\sqrt{p}\rf$ and $p$ is an integer, $n^2 \leq p \leq n^2+2n$, so these four cases include all possible values of $p$.
\bigskip

\noindent \textbf{Case 1. $p\in[n^2,n^2+\nf]$.} Note that if $a=b=n$ then $g(a,b)=2/n$.  These values of $a$ and $b$ satisfy $ab \leq p$ since $ab=n^2\leq p$. Now assume by way of contradiction that there exist integers $c$ and $d$ with $1 \leq c \leq d \leq p$, $cd \leq p$, and $g(c,d)< 2/n$.

\textbf{Subcase 1a. $c,d\geq n$.} Since $c$ and $d$ are not both $n$, at least one of them must be at least
$n+1$, in which case $$cd\geq n(n+1)=n^2+n>n^2+\left\lfloor\frac{n}{2}\right\rfloor\geq
p,$$ contradicting $cd \leq p$.

\textbf{Subcase 1b. $c,d\leq n$.} Since $c$ and $d$ are not both $n$, $g(c,d)\geq 2/n$.

\textbf{Subcase 1c. $c<n$ and $d>n$.} We write $c=n-j$ and $d=n+k$ with $j,k>0$.  By
assumption,
$$g(c,d)=\frac{1}{n-j}+\frac{1}{n+k}<\frac{2}{n}\quad\quad\text{and }\quad\quad (n-j)(n+k)\leq p.
$$ Clearing denominators in the first inequality and then substituting
in the second gives
$$
n(n+k)+n(n-j)<2(n-j)(n+k)\leq 2p\leq 2\left(n^2+\left\lfloor\frac{n}{2}\right\rfloor\right).
$$
\noindent Cancelling $2n^2$ from the first and and last expressions gives
$$nk-nj<2\left\lfloor\frac{n}{2}\right\rfloor\leq n,$$ so $k-j<1$ and $k\leq j$.  If $k=j$ then $$
g(c,d)=\frac{1}{c}+\frac{1}{d}=\frac{1}{n-j}+\frac{1}{n+j}=\frac{2n}{n^2-j^2}=\frac{2}{n-\frac{j^2}{n}}>\frac{2}{n},
$$ contradicting $g(c,d)<2/n$.  If $k<j$, then $g(c,d)>1/(n-j)+1/(n+j)>\frac{2}{n}$, still contradicting $g(c,d)<2/n$.
\bigskip

\noindent \textbf{Case 2. $p\in(n^2+\llf\frac{n}{2}\rrf,n^2+2\llf\frac{n}{2}\rrf]$.} Assume by way of contradiction that there exist integers $c$ and $d$ with $1 \leq c \leq d \leq p$, $cd \leq p$, and $g(c,d)<\frac{2n+1}{p}$.

\textbf{Subcase 2a. $c,d\leq n$.} Then
$$
\frac{1}{c}+\frac{1}{d}\geq \frac{2}{n}\geq \frac{2(n^2+\llf\frac{n}{2}\rrf+1)}{np}\geq \frac{2n^2+n}{np}=\frac{2n+1}{p},
$$ contradicting $g(c,d)<(2n+1)/p$.

\textbf{Subcase 2b. $c,d\geq n+1$.} Then
$$
cd\geq (n+1)^2=n^2+2n+1> n^2+2\llf\frac{n}{2}\rrf\geq p,
$$ contradicting $cd \leq p$.

\textbf{Subcase 2c. $c\leq n$ and $d> n$.} Thus we can write $c=n-j$, $d=n+k$ with $j\geq 0$, $k\geq 1$. Since $1/c+1/d< (2n+1)/p$, we have $p(c+d)<(2n+1)cd\leq p(2n+1)$, so $c+d\leq 2n$, which means $k\leq j$.  So, in particular, $j \geq 1$.  If $j=k$, then we have
$$
\frac{1}{c}+\frac{1}{d}=\frac{1}{n-j}+\frac{1}{n+j}=\frac{2n}{n^2-j^2}\geq\frac{2n+1}{p},
$$
the last line resulting from cross-dividing the following inequality:
$$
2np>2n\left(n^2+\llf\frac{n}{2}\rrf \right)\geq2n\left(n^2+\left(\frac{n}{2}-\frac{1}{2}\right)\right)=2n^3+n^2-n>2n^3+n^2-j^2(2n+1)=(2n+1)(n^2-j^2).
$$ If $k<j$, then $g(c,d)>1/(n-j)+1/(n+j)\geq (2n+1)/p$, still contradicting $g(c,d)<(2n+1)/p$.
\bigskip

\noindent \textbf{Case 3. $p\in(n^2+2\llf n/2 \rrf,n^2+n+\llf n/2 \rrf ]$.} Note that if $a=n$ and $b=n+1$ then $g(a,b)=1/n+1/(n+1)$.  These values of $a$ and $b$ satisfy $ab \leq p$ since $ab=n(n+1)\leq n^2+2\llf\frac{n}{2}\rrf+1\leq p$.  Now assume by way of contradiction that there exist integers $c$ and $d$ with $1 \leq c \leq d \leq p$, $cd \leq p$, and $g(c,d)< 1/n+1/(n+1)$.

\textbf{Subcase 3a. $c\leq n$ and $d\leq n+1$.} Then $g(c,d)\geq\frac{1}{n}+\frac{1}{n+1}$.

\textbf{Subcase 3b. $c\geq n$ and $d\geq n+2$, or $c, d \geq n+1$.} Then
$cd\geq n^2+2n> n^2+n+\left\lfloor n/2 \right\rfloor\geq p$, contradicting $cd \leq p$.

\textbf{Subcase 3c. $c<n$ and $d>n+1$.} We write $c=n-j$ and $d=n+k$ with $j>0$, $k>1$.  Clearing the denominators of
\begin{align*}
\frac{1}{n-j}+\frac{1}{n+k}<\,&\frac{1}{n}+\frac{1}{n+1},
\end{align*}
and distributing the left-hand side now gives
\begin{align*}
2n^3+(k-j+2)n^2+(k-j)n&<(n+1)(n-j)(n+k)+n(n-j)(n+k)\\
&\leq(n+1)p+np\\
&\leq (2n+1)\left(n^2+n+\llf\frac{n}{2}\rrf\right)\\
  &=2n^3+3n^2+2n\llf\frac{n}{2}\rrf+n+\llf\frac{n}{2}\rrf\\
&\leq 2n^3+4n^2+2n.
\end{align*}
Subtracting $2n^3+2n^2$ from both sides and dividing by $n^2+n$ gives
$k-j<2$, and so $k\leq j+1$. If $k=j+1$ then
$$
g(c,d)=\frac{1}{n-j}+\frac{1}{n+(j+1)}=\frac{2n+1}{n^2-j^2+n-j}>\frac{2n+1}{n^2+n}=\frac{1}{n}+\frac{1}{n+1},
$$ contradicting our assumption that $g(c,d)<1/n+1/(n+1)$. If $k<j+1$, then $g(c,d)>1/(n-j)+1/(n+(j+1))\geq 1/n + 1/(n+1)$, still contradicting $g(c,d)<1/n+1/(n+1)$.
\bigskip

\noindent \textbf{Case 4. $p\in (n^2+n+\llf\frac{n}{2}\rrf,n^2+2n]$.}  Assume by way of contradiction that there exist integers $c$ and $d$ with $1 \leq c \leq d \leq p$, $cd \leq p$, and $g(c,d)<\frac{2n+2}{p}$.

\textbf{Subcase 4a. $c\leq n$ and $d\leq n+1$.} Then
$$
\frac{1}{n}+\frac{1}{n+1}=\frac{(2n+1)(2n+2)}{(n^2+n)(2n+2)}
>\frac{(2n+2)(2n+1)}{(n^2+\frac{3n}{2}+\frac{1}{2})(2n+1)}\geq\frac{2n+2}{n^2+n+\llf\frac{n}{2}\rrf+1}\geq\frac{2n+2}{p},
$$
contradicting $g(c,d)<(2n+2)/p$.

\textbf{Subcase 4b. $c,d\geq n+1$.} Then
$$
cd\geq n^2+2n+1>n^2+2n\geq p,
$$ contradicting $cd \leq p$.

\textbf{Subcase 4c. $c\leq n$ and $d> n+1$.}  Write $c=n-j$ and $d=n+k$, with $j\geq0$
and $k> 1$.  Since $g(c,d)=1/c+1/d< (2n+2)/p$, it follows that
$p(d+c)<(2n+2)cd\leq (2n+2)p$, so $d+c=2n-j+k<2n+2$, giving $k\leq
j+1$.  If $k=j+1\geq 2$, then the bound
$$
(2n+1)p\geq (2n+1)\left(n^2+n+\llf\frac{n}{2}\rrf +1 \right)\geq 2n^3+4n^2+ \frac{5}{2} n+\frac{1}{2}
$$
$$\geq 2n^3+4n^2+(2-2j-2j^2)n-2(j+j^2)=(2n+2)(n-j)(n+j+1)
$$
gives
$$
\frac{1}{c}+\frac{1}{d}=\frac{1}{n-j}+\frac{1}{n+j+1}=\frac{2n+1}{(n-j)(n+j+1)}\geq\frac{2n+2}{p},
$$ contradicting our assumption that $g(c,d)< (2n+2)/p$.  Finally, if $k<j+1$, then $g(c,d)>1/(n-j)+1/(n+(j+1))\geq (2n+2)/p$, still contradicting $g(c,d)<(2n+2)/p$.
\end{proof}

Table \ref{table}, at the end of the paper, gives a list of values for the pebbling weight function for paths of length between 2 and 7 and for varying values of $p$.  The values were obtained through a computer search.   

\section{Graphs Requiring Specific Edge Weights}

Recall that the weighted pebbling number of a graph $G$ is the smallest weighted pebbling number of all weighted graphs $G_W$ with $|W|=E(G)/2$.    Although this minimum is taken over all possible weight distributions on $G$, in previous examples the smallest number of pebbles was obtained with edge weights of $1/2$, $1$, or $0$.  We ask whether this is always the case.   In other words, given a rational number $a$ with $0 \leq a \leq 1$, is there a graph which requires an edge weight $a$ to achieve its weighted pebbling number?  Note by Lemma \ref{rational}, edge weights may always be taken to be rational.  
In  Theorem~\ref{thm:1/3-required}, we prove that the tree $T$ shown in Figure~\ref{1/3-tree} requires an edge weight of $1/3$.  For all rational numbers in $(0,1)$ other than 1/2 and 1/3, this question remains open.

\begin{figure} [!htbp]
\begin{center}
 \includegraphics[width=6cm]{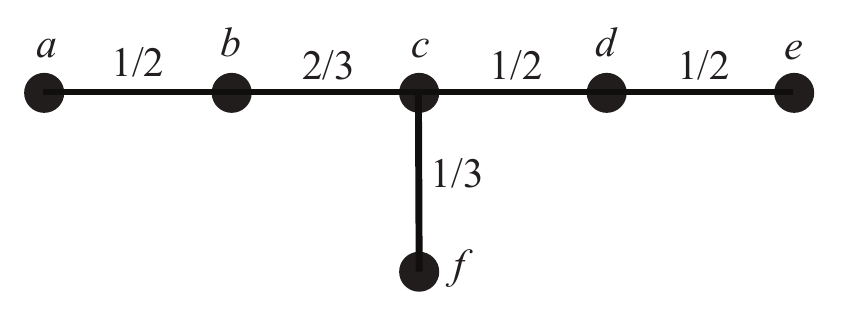}
  \end{center} \caption{A weight distribution $\widehat{W}$ on $T$ with $wp(T_{\widehat{W}})=15$.} \label{1/3-tree}
\end{figure}

We begin by establishing the weighted pebbling number of $T$.  Since $T$ is a tree, the weighted pebbling number of each weight distribution on $T$ is given by Theorem~\ref{BigTheorem}.  However, the formula in Theorem~\ref{BigTheorem} depends on a maximum path partition of $T$, which in turn depends on its weight distribution.

\begin{proposition}  Let $T$ be the tree in Figure~\ref{1/3-tree}.  Then
$wp(T)=15$.
\label{thm:1/3-tree}
\end{proposition}

\begin{proof}
Figure~\ref{1/3-tree} gives a weight distribution $\widehat{W}$ on $T$ for which $wp(T_{\widehat {W}})=15$ by Theorem~\ref{BigTheorem}, and thus $wp(T) \leq 15$.  Now suppose by way of contradiction that there exists a weight distribution $W$ on $T$ for which $wp(T_W)\leq 14$ and $|W|=5/2$.  So all distributions of 14 pebbles on $T_W$ can be solved to any target.  We consider the requirement of the path $P$ from $a$ to $e$.  Since by Table~\ref{table}, any path of length $4$ with requirement less than $8$ has total weight more than $5/2$, we know $r(P)\geq 8$.  Also, since $w(cf)>0$ (else a configuration with 14 pebbles on $f$ and target $c$ is unsolvable), the total weight of $P$ must be less than $5/2$, so, in fact, $r(P)\geq 9$.  On the other hand, since $wp(T_W)\leq 14$ by assumption, $r(P)\leq 14$.  Thus, there are 6 cases to consider.  In what follows, we use the values from Table~\ref{table} for the pebbling weight function $w_{P_4}(p)$.
Suppose $r(P)=k$ for some integer $9 \leq k \leq 14$.  Then the total weight on $P$ must be at least $w_{P_4}(k)$, leaving at most $5/2-w_{P_4}(k)$ weight on the remaining edge $cf$.
\bigskip

\noindent \textbf{Case 1. $r(P)=9$.} Then $w(cf) \leq 1/18$, so the pebbling configuration with target $e$ and $14$ pebbles on $f$ is not solvable.
\bigskip

\noindent \textbf{Case 2. $r(P)=10$.}  Then $w(cf) \leq 1/6$, so the pebbling configuration with target $e$, $9$ pebbles on $a$, and $5$ pebbles on $f$ is not solvable.
\bigskip

\noindent \textbf{Case 3. $r(P)=11$.}  Then $w(cf) \leq 8/33$, so the pebbling configuration with target $e$, $10$ pebbles on $a$, and $4$ pebbles on $f$ is not solvable.
\bigskip

\noindent \textbf{Case 4. $r(P)=12$.} Then $w(cf) \leq 1/3$.  If $w(cf) < 1/3$, then the configuration with target $e$, 11 pebbles on $a$, and 3 pebbles on $f$ is not solvable.  So, assume $w(cf) = 1/3$, and hence $w(P) = 13/6$.  Consider the two undirected paths $Q=abc$ and $Q'=cde$.  Assume without loss of generality that $w(Q) \leq w(Q')$.  Thus, $w(Q) \leq 13/12$ and since a weight of 7/6 is required for a directed path of length 2 to be 3-solvable (by computer search), 3 pebbles on $c$ cannot be pebbled to $a$.  

\textbf{Subcase 4a.  $13/12 \leq w(Q') < 7/6$.}
The pebbling configuration with 11 pebbles on $f$, 3 pebbles on $e$, and target vertex $a$ is not solvable, since only 3 pebbles will arrive at $c$ and, thus, no pebbles will get to $a$.  

\textbf{Subcase 4b.  $7/6 \leq w(Q') \leq 4/3$.}
There is not enough weight on $Q'$ to be 3-solvable in both directions (see Table \ref{table}).  If the directed path from $e$ to $c$ is not 3-solvable, then the configuration with 11 pebbles on $f$ and 3 on $e$ is again not solvable to $a$.  If, on the other hand, the directed path from $c$ to $e$ is not solvable, then the configuration with target $e$, 3 pebbles on $a$, and 11 pebbles on $f$ is not solvable.  

\textbf{Subcase 4c.  $w(Q') \geq 4/3$.}
Then $w(Q) \leq 5/6$.  Since $Q$ is not 4-solvable in either direction (by computer search), the configuration with 14 pebbles on $f$ and target $a$ is not solvable.  
\bigskip

\noindent \textbf{Case 5. $r(P)=13$.} Then $w(cf) \leq 1/3$, so the pebbling configuration with target $e$, $12$ pebbles on $a$, and $2$ pebbles on $f$ is not solvable.
\bigskip

\noindent \textbf{Case 6. $r(P)=14$.} Then $w(cf) \leq 5/14$, so the pebbling configuration with target $e$, $13$ pebbles on $a$, and $1$ pebble on $f$ is not solvable.

\end{proof}

In the next proposition we show that any weight distribution for which $T$ is 15-solvable includes an edge with weight 1/3.  In this case we say that an edge weight of $1/3$ is \textbf{required} for $T$ to achieve its weighted pebbling number.  
The proof of the proposition again uses minimum weights required for pebbling across paths, found by a computer search.  Many of these values appear in Table \ref{fourpebbles}, which gives the weight needed to move $p$ pebbles along a path of length 2 in one direction and arrive at the destination with $k$ pebbles, and Table \ref{table}, which gives the weights required for a path of length $n$ to be $p$-solvable.

\begin{table}[htdp]
\caption{Weight required to move $p$ pebbles along a one-way path of length 2 finishing with $k$ pebbles}
\begin{center}
\begin{tabular}{c|c|c|c|c}
$p \backslash k$ &3&  4 & 5 & 6\\
\hline
8 & 49/40 & 17/12 & 19/12 & 97/56 \\
9 & 52/45 & 4/3 & 94/63 & 103/63 \\
10& 11/10 & 19/15 & 99/70 & 31/20 \\
11& 23/22 & 93/77 & 104/77 & 65/44 \\
12 & 1& 97/84 & 31/24 & 17/12 \\
13 & 25/26 & 101/91 & 129/104 & 53/39 \\
14 & 13/14 & 15/14 & 67/56 & 55/42 \\
15 & 94/105 & 31/30 & 52/45 & 19/15 \\
\end{tabular}
\end{center}
\label{fourpebbles}
\end{table}

\begin{theorem}
The tree $T$ in Figure \ref{1/3-tree} requires an edge weight of 1/3 to achieve its weighted pebbling number.
\label{thm:1/3-required}
\end{theorem}

\begin{proof}
The total weight on the edges of $T$ is 5/2.
We prove that the edge $cf$ requires a weight of 1/3 by contradiction.  Suppose $w(cf)$ is not 1/3.  Without loss of generality, suppose that $ w(ab) + w(bc) \leq w(cd) + w(de)$.
\bigskip

\noindent \textbf {Case 1.} $w(cf) \geq 1/2$.  Then the combined edge weights on the path of length 4, $abcde$ is at most 2 and, by Table \ref{table}, this is not enough weight to be solvable with 15 pebbles.  So either the pebbling configuration with target $e$ and 15 pebbles on $a$ or with target $a$ and 15 pebbles on $e$ is not solvable. 

\noindent {\bf Case 2.} $1/3 < w(cf) < 1/2 $.  In this case, the combined weight on path $abcde$ is at most 13/6, and so the pebbling configuration with target $e$, 13 pebbles on $a$ and 2 pebbles on $f$ (or target $a$, 13 pebbles on $e$ and 2 pebbles on $f$) is not solvable, again by Table \ref{table}.

\noindent {\bf Case 3.} $1/4 \leq w(cf) < 1/3 $.  To solve to target $f$, we need to get at least 4 pebbles to vertex $c$.  We show that regardless of the remaining weight distribution, there exists an unsolvable configuration of pebbles.  Note that $w(ab) + w(bc) \leq 9/8$ and $w(cd) + w(de) \geq 9/8$.

{\bf Subcase 3a.}   $9/8 \leq w(cd) + w(de) < 7/6$.  The pebbling configuration with target $f$, 12 pebbles on vertex $a$ and 3 pebbles on vertex $e$ is not solvable, because the weight on path $edc$ is less than the 7/6 required for this one-way path to be 3-solvable (by computer search), and hence no pebbles can get from $e$ to $c$,
 and the weight on path $abc$ is less than the 97/84 required to move 4 pebbles from vertex $a$ to vertex $c$ (see Table \ref{fourpebbles}).

{\bf Subcase 3b.}  $7/6 \leq w(cd) + w(de) < 3/2$.  (Thus, $w(ab) + w(bc) \leq 13/12$).  The pebbling configuration with 2 pebbles on $e$ and 13 pebbles on $a$ is not solvable to $f$, since it requires a weight of at least 3/2 to move one pebble from vertex $e$ to vertex $c$ (by computer search) and a weight of at least 101/91 to move 4 pebbles from vertex $a$ to vertex $c$ (see Table \ref{fourpebbles}).

{\bf Subcase 3c.}  $3/2 \leq w(cd) + w(de) \leq 9/4 $.  The pebbling configuration with target $f$ and 15 pebbles on vertex $a$ is not solvable, since $w(ab) + w(bc) \leq 3/4 $, but a weight of at least $31/30$ is required to move 4 pebbles to vertex $c$ (see Table \ref{fourpebbles}).

\noindent {\bf Case 4.}  $1/5 \leq w(cf) < 1/4 $.  The pebbling configuration with target $f$ and 15 pebbles on $a$ is not solvable since the path $abc$ has a combined weight of at most $23/20$, but, by Table \ref{fourpebbles}, it requires a weight of at least $52/45$ to move 5 pebbles from vertex $a$ to vertex $c$ and hence to move 1 pebble to $f$.

\noindent{\bf Case 5.} $w(cf) < 1/5$.  Place 15 pebbles on vertex $a$, and let $f$ be the target vertex.  The combined weight of path $abc$ is at most 5/4.  However, by Table \ref{fourpebbles}, it requires a weight of at least 19/15 to move 6 pebbles to vertex $c$.  Therefore this pebbling configuration is not solvable to $f$.
\end{proof}

We remark that for some graphs $G$, $|E(G)|/2$ is more than enough weight to solve any distribution with $wp(G)$ pebbles.  For example, $wp(K_5)=1$ and $|E(K_5)|/2=5$, but there exists a weight distribution $W$ on $K_5$ of total weight 4 with weighted pebbling number 1.  So the extra weight of 1 is unnecessary.  
By contrast, the proof of Proposition~\ref{thm:1/3-required} implies that the full weight $|E(T)|/2$ is needed for $T$ to achieve its weighted pebbling number.  This is also the case for paths and stars by Corollary \ref{PathPebblingNumber} and Proposition \ref{starpf}.  We conjecture the following. 
\begin{conjecture} \label{full weight}
All trees require the full weight $|E(G)|/2$ to achieve their weighted pebbling number. 
That is, for a tree $T$, $w_T(wp(T))=|E(G)|/2$.
\end{conjecture} 

\section{Longer Paths and Further Directions}

In this section, we observe an upper bound for the pebbling weight function of the 3-edge path $P_3$, and we give a lower bound for the pebbling weight function of a path of length $n$.

\begin{remark}
We have 
$$ w_{P_3}(p) \leq 
\min_k \left[\frac{\lceil \frac{p}k \rceil}{p} + \frac k{\lceil \frac{p}k \rceil} +  \frac{\lceil \frac{p}k \rceil}{p}\right].
$$
\end{remark}

Let $P_3=abcd$, and assign weights
$\frac{\lceil \frac{p}k \rceil}{p}$, $\frac k{\lceil \frac{p}k \rceil}$ and $\frac{\lceil \frac{p}k \rceil}{p}$ to the three edges $ab,bc,$ and $cd$ respectively.  Notice that if we place $p$ pebbles on $a$ with target $d$, we are able to move $\lceil \frac{p}k \rceil$ pebbles to $b$, $k$ pebbles to $c$, and (since $(k \lceil \frac{p}k \rceil)/p \geq 1$), at least one pebble to $d$.  Thus, the minimum of
$\min_k \left (\frac{\lceil \frac{p}k \rceil}{p} + \frac k{\lceil \frac{p}k \rceil} +  \frac{\lceil \frac{p}k \rceil}{p} \right)$
 over $k$ is an upper bound for the pebbling weight function of $P_3$.  
We further note that when $$k = \left\lceil \frac16\left(108p+3\sqrt{-3+1296p^2}\right)^{1/3} + \frac1{2\left(108p+3\sqrt{-3+1296p^2}\right)^{1/3}} -\frac12 \right\rceil,$$
obtained by solving the equation
$p=k^3+(3/2)k^2+(1/2)k$ for $k$,
the above weight distribution gives an exact bound on $wp(P_3)$ approximately 60\% of the time in the range $1\leq p \leq 100$.  In the cases where the bound is not tight, for $p \leq 100$, the actual weight distribution is within 0.05 of the bound.

\begin{proposition} \label{pathlb}
$w_{P_n}(p) \geq n/\sqrt[n]{p}$.
\end{proposition}

\begin{proof}
Suppose $W$ is a weight distribution for which $P_n$ is $p-$solvable.  Let $v_1, \ldots, v_{n+1}$ be the vertices of $P_n$, and let $w_{i}$ be the weight on edge $v_iv_{i+1}$.
Since the weighted graph $(P_n)_W$ must be solvable for any distribution of pebbles to any target, consider the distribution with $p$ pebbles on $v_1$ and target $v_{n+1}$.  The only solution involves moving all $p$ pebbles from $v_1$ toward $v_{n+1}$.  Now $\lfloor pw_1\rfloor$ pebbles arrive at $v_2$, $\lfloor \lfloor pw_1\rfloor \cdot w_2 \rfloor$ pebbles arrive at $v_3$, etc., and since at least one pebble arrives at $v_{n+1}$, $\lfloor \ldots \lfloor pw_1\rfloor \cdot w_2 \rfloor  \ldots \cdot w_{n}\rfloor \geq 1$.  Thus $ p w_1 w_2 \cdots w_{n} \geq 1$.  The minimum of $\sum_{i=1}^{n}w_i$ subject to the constraint $\Pi_{i=1}^n w_i \geq 1/p$ occurs when $w_1 = w_2 = \cdots = w_n = 1/\sqrt[n]{p}$.  Thus $|W| \geq n/\sqrt[n]{p}$ and the result follows.
\end{proof}

This lower bound is very close to the known values of $w_{P_n}(p)$, and exact when $\sqrt[n]p$ is an integer. Figure \ref{maple} shows this bound plotted against the exact values from Table~\ref{table}.

\begin{figure} [!htbp]
\begin{center}\includegraphics[width=7cm]{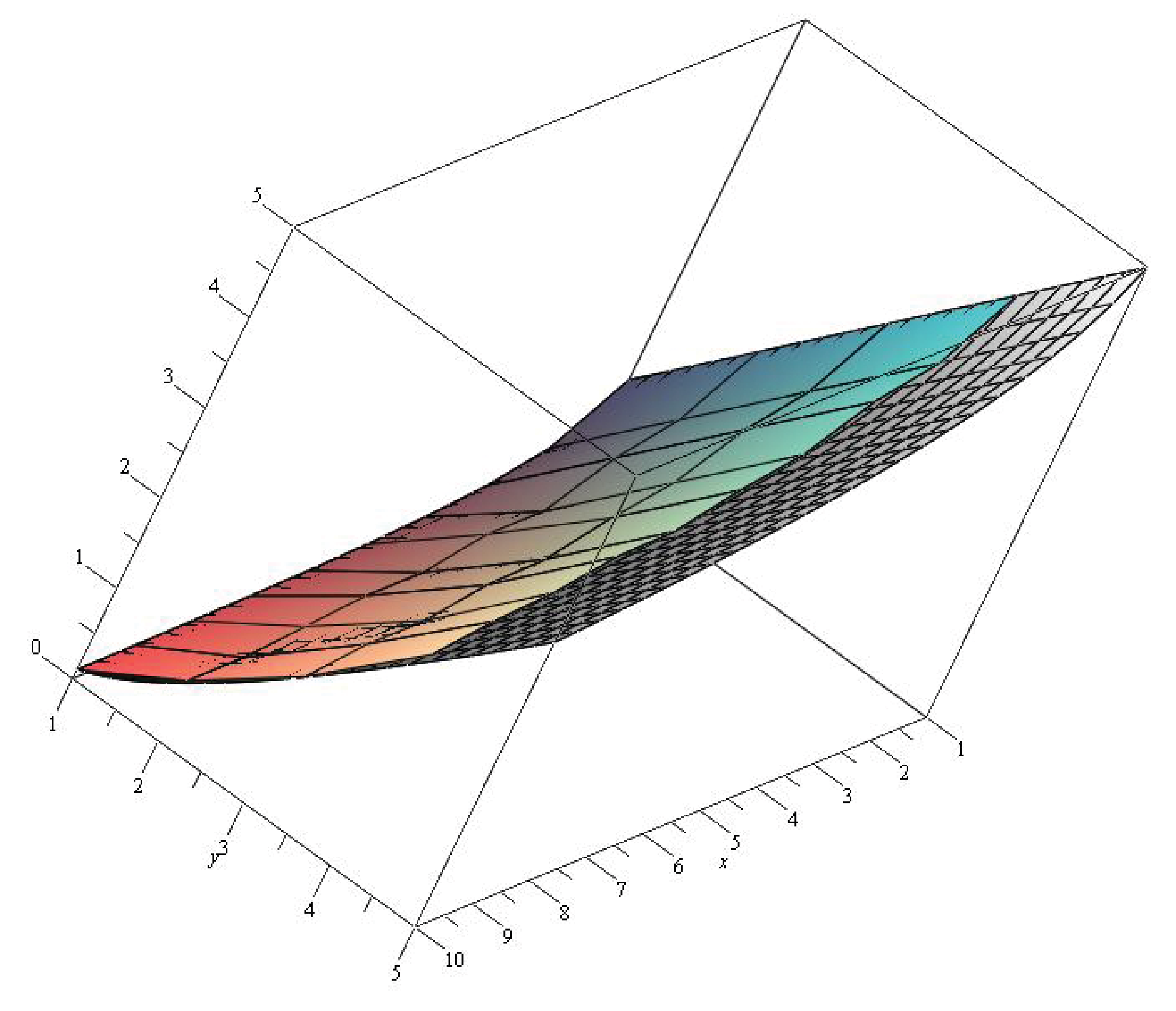}
 \end{center}\caption{Comparison of calculated minimum weight $w_{P_n}(p)$ to the lower bound of Proposition \ref{pathlb} for $1 \leq p \leq 10$ and $1 \leq n \leq 5$.} \label{maple}
\end{figure}

Many questions remain open.  For example, it would be interesting to find a formula for the weighted pebbling number of trees that does not depend on the target vertex or weight distribution.  Other questions include:
\begin{enumerate}
\item What is a formula for the pebbling weight function for general paths or graphs?
\item For which rational numbers
$a$, $0\leq a \leq 1$, are there graphs which require an edge weight of $a$ to achieve their weighted pebbling number?
\item What is the weighted pebbling number of the cycle $C_n$?
\item Can we classify graphs by their weighted pebbling number?  In particular, in light of Proposition \ref{2n-2}, what graphs $G$ have $wp(G)=2$, etc.?  
\end{enumerate}

We conclude with a table of computer-verified weight requirements for paths of varying lengths and for different numbers of pebbles.  
The discrepancies we see in the denominators of values in the table are curious.  For example, we need a minimum weight of 2 to move 18 pebbles across a path of length 4, but to move 19 across we need a weight of 371/190.
These numbers are approximating $\sqrt[n]p$ where $p$ is the number of pebbles and $n$ is the length of the path, with rational numbers, but it is unclear why some approximations have large error and others very small.  This is also an interesting direction for future research.

\begin{table}[h!]
\caption{Values of $w_{G}(p)$ if $G$ is a path with $n$ edges.} \label{table}
\begin{tabular}{c|c|c|c|c|c|c}
  $p \backslash n$  & 2 & 3 & 4 & 5 & 6 & 7 \\
\hline
1 & 2 & 3 & 4 & 5 & 6 & 7 \\
2 & 3/2 & 5/2 & 7/2 & 9/2 & 11/2 & 13/2 \\
3 & 4/3 & 7/3 & 10/3 & 13/3 & 16/3 & 19/3 \\
4 & 1 & 2 & 3 & 4 & 5 & 6 \\
5 & 1 & 28/15 & 43/15 & 58/15 & 73/15 & 88/15 \\
6 & 5/6 & 5/3 & 8/3 & 11/3 & 14/3 & 17/3 \\
7 & 5/6 & 23/14 & 37/14 & 51/14 & 65/14 & 79/14 \\
8 & 3/4 & 3/2 & 5/2 & 7/2 & 9/2 & 11/2 \\
9 & 2/3 & 3/2 & 22/9 & 41/12 & 53/12 & 65/12 \\
10 & 2/3 & 7/5 & 7/3 & 197/60 & 257/60 & 317/60 \\
11 & 7/11 & 7/5 & 149/66 & 419/132 & 551/132 & 683/132 \\
12 & 7/12 & 4/3 & 13/6 & 37/12 & 49/12 & 61/12 \\
13 & 7/12 & 4/3 & 13/6 & 37/12 & 49/12 & \\
14 & 4/7 & 9/7 & 15/7 & 106/35 & 141/35 & \\
15 & 8/15 & 53/42 & 31/15 & 3 & 4 & \\
16 & 1/2 & 37/30 & 2 & 44/15 & 59/15 & \\
17 & 1/2 & 41/34 & 2 & 295/102 & 397/102 & \\
18 & 1/2 & 7/6 & 2 & 17/6 & 23/6 & \\
19 & 9/19 & 155/133 & 371/190 & 17/6 & 289/76 & \\
20 & 9/20 & 79/70 & 19/10 & 14/5 & 15/4 & \\
21 & 9/20 & 23/21 & 19/10 & 293/105 & 1147/308 & \\
22 & 9/20 & 23/21 & 19/10 & 91/33 & 485/132 & \\
23 & 10/23 & 197/184 & 43/23 & 187/69 & 1001/276 & \\
24 & 5/12 & 25/24 & 11/6 & 8/3 & & \\
25 & 2/5 & 25/24 & 11/6 & 8/3 & & \\
26 & 2/5 & 40/39 & 9/5 & 241/91 & & \\
27 & 2/5 & 1 & 97/54 & 493/189 & & \\
28 & 11/28 & 1 & 16/9 & 18/7 & & \\
29 & 11/29 & 287/290 & 511/290 & 18/7 & & \\
30 & 11/30 & 29/30 & 26/15 & 77/30 & & \\
31 & 11/30 & 29/30 & 26/15 & 15//62 & & \\
32 & 11/30 & 169/176 & 55/32 & 5/2 & & \\
33 & 4/11 & 31/33 & 41/24 & 5/2 & & \\
34 & 6/17 & 31/33 & 403/238 & 5/2 & & \\
35 & 12/35 & 14/15 & 2003/1190 & & &  \\
36 & 1/3 & 11/12 & 349/210 & & &  \\
37 & 1/3 & 11/12 & & & & \\
38 & 1/3 & 451/494 & & & & \\
39 & 1/3 & 35/39 & & & & \\
40 & 13/40 & 139/156 & & & & \\

\end{tabular}
\end{table}

\bibliographystyle{plain}
\bibliography{bibliography}

\begin{thebibliography}{10}

\bibitem{bekmetjev09}
Airat Bekmetjev and Charles~A. Cusack.
\newblock Pebbling algorithms in diameter two graphs.
\newblock {\em SIAM J. Discrete Math.}, 23(2):634--646, 2009.

\bibitem{bukh06}
Boris Bukh.
\newblock Maximum pebbling number of graphs of diameter three.
\newblock {\em J. Graph Theory}, 52(4):353--357, 2006.

\bibitem{chung89}
Fan R.~K. Chung.
\newblock Pebbling in hypercubes.
\newblock {\em SIAM J. Discrete Math.}, 2(4):467--472, 1989.

\bibitem{clarke97}
T.~A. Clarke, R.~A. Hochberg, and G.~H. Hurlbert.
\newblock Pebbling in diameter two graphs and products of paths.
\newblock {\em J. Graph Theory}, 25(2):119--128, 1997.

\bibitem{curtis09}
Dawn Curtis, Taylor Hines, Glenn Hurlbert, and Tatiana Moyer.
\newblock On pebbling graphs by their blocks.
\newblock {\em Integers}, 9:G02, 411--422, 2009.

\bibitem{feng01}
Rongquan Feng and Ju~Young Kim.
\newblock Graham's pebbling conjecture on product of complete bipartite graphs.
\newblock {\em Sci. China Ser. A}, 44(7):817--822, 2001.

\bibitem{herscovici03}
David~S. Herscovici.
\newblock Graham's pebbling conjecture on products of cycles.
\newblock {\em J. Graph Theory}, 42(2):141--154, 2003.

\bibitem{herscovici98}
David~S. Herscovici and Aparna~W. Higgins.
\newblock The pebbling number of {$C_5\times C_5$}.
\newblock {\em Discrete Math.}, 187(1-3):123--135, 1998.

\bibitem{hurlbert05}
Glenn Hurlbert.
\newblock Recent progress in graph pebbling.
\newblock {\em Graph Theory Notes N.~Y.}, 49:25--37, 2005.

\bibitem{hurlbert99}
Glenn~H. Hurlbert.
\newblock A survey of graph pebbling.
\newblock In {\em Proceedings of the {T}hirtieth {S}outheastern {I}nternational
  {C}onference on {C}ombinatorics, {G}raph {T}heory, and {C}omputing ({B}oca
  {R}aton, {FL}, 1999)}, volume 139, pages 41--64, 1999.

\bibitem{wang09}
Zhiping Wang, Yutang Zou, Haiying Liu, and Zhongtuo Wang.
\newblock Graham's pebbling conjecture on product of thorn graphs of complete
  graphs.
\newblock {\em Discrete Math.}, 309(10):3431--3435, 2009.

\end{thebibliography}

\end{document}